\documentclass{amsart}
\usepackage{amsmath}
\usepackage{amssymb}
\usepackage{amsthm}
\usepackage{enumerate}
\usepackage[dvips]{graphicx}
\theoremstyle{definition}
\newtheorem{definition}{Definition}[section]

\newtheorem{example}[definition]{Example}
\theoremstyle{plain}

\newtheorem{theorem}[definition]{Theorem}
\newtheorem{proposition}[definition]{Proposition}

\newtheorem{fact}[definition]{Fact}

\newcommand{\bN}{\mathbb N}

\newcommand{\cM}{\mathcal M}

\newcommand{\ov}[1]{\overline{#1}}

\def \M {{\mathcal M}}

\newcommand{\N}{\mathbb{N}}

\def \sup {{\rm sup}}

\def \< {{\langle}}
\def \> {{\rangle}}

\DeclareMathOperator{\Int}{int}
\DeclareMathOperator{\cl}{cl}

\makeatletter
\@namedef{subjclassname@2020}{
  \textup{2020} Mathematics Subject Classification}
\makeatother

\begin{document}

\title[Open cell property in weakly o-minimal structures]{Open cell property in weakly o-minimal structures}
\author[T. Kawakami]{Tomohiro Kawakami}
\address{Department of Mathematics,
	Wakayama University,
	Wakayama, 640-8510, Japan}
\email{kawa0726@gmail.com}

\author[H. Tanaka]{Hiroshi Tanaka}
\address{Faculty of Engineering, Kindai University,
Hiroshima, 739-2116, Japan}
\email{htanaka@hiro.kindai.ac.jp}

\begin{abstract}
	Every bounded definable open set is a union of finitely many
	open strong cells in a weakly o-minimal expansion of a real
	closed field. We prove this fact and another theorem similar
	to it.	
\end{abstract}


\keywords{Weakly o-minimal structures, open strong cells and open cell property.
The first author was partially supported by JSPS
KAKENHI Grant Number 23K03095.
}

\maketitle

\section{Introduction}

Throughout this paper, 
``definable'' means ``definable possibly with parameters''
and 
we assume that 
a structure $\mathcal M = (M, <, \dots)$ is 
a dense linear order without endpoints. 

A subset $A$ of $M$ is said to be convex
if $a, b \in A$ and $c \in M$ with $a < c < b$ then $c \in A$.
Moreover if 
$A = \emptyset$ or $\inf A, \sup A \in M \cup \{-\infty, +\infty\}$,
then $A$ is called an interval in $M$.
We say that $\mathcal M$ is o-minimal (weakly o-minimal)
if every definable subset of $M$ is 
a finite union of intervals (convex sets), respectively. 
Weakly o-minimality was introduced by Dickmann in \cite{Di}.

Wilkie  (\cite{Wi}) proved that in o-minimal expansions of a real closed field
every bounded definable open set is a finite union of open cells.
Andrews (\cite{A}) proved that if an o-minimal structure admits
CE-cell decomposition then any definable open set is expressed as
a finite union of definable open cells.
Edmundo et al (\cite{EEP}) proved in semi-bounded o-minimal expansion of an ordered group,
any definable open set is covered by a finite union of open cells.
The relation between strong cell decomposition in weakly o-minimal structures and 
the open cell property in o-minimal structures is proved by \cite{ET}.
We consider a generalization of the above results in weakly o-minimal structures.

Readers are expected to be familiar with 
fundamental results of o-minimality and weak o-minimality; 
see, for example, \cite{C},  \cite{D}, \cite{MMS}, \cite{W1}.

\section{Weakly o-minimal expansions of real closed fields}


Let $\mathcal M = (M, <, \ldots)$ be a weakly o-minimal expansion of a dense linear order without endpoints.
For any subsets $C, D$ of $M$,
we write $C < D$ if $c < d$ whenever $c \in C$ and $d \in D$.
A pair $\langle C, D \rangle$ of non-empty subsets of $M$ is called
a cut in $\mathcal M$
if $C < D, C \cup D = M$ and $D$ has no lowest element.
A cut $\langle C, D \rangle$ is said to be definable in $\cM$
if the sets $C$ and $D$ are definable in $\cM$.
The set of all cuts definable in $\cM$ will be denoted by $\ov{M}$.
Note that we have $M = \ov{M}$ if $\cM$ is o-minimal.
We define a linear ordering on $\overline{M}$ by 
$\langle C_1, D_1 \rangle < \langle C_2, D_2 \rangle$ 
if and only if $C_1 \subsetneq C_2$.
Then we may treat $(M, <)$ as a substructure of $(\overline{M}, <)$
by identifying an element $a \in M$ with 
the definable cut $\langle (-\infty, a], (a, +\infty) \rangle$.
Moreover, if $\mathcal M = (M, <, +, \ldots)$ is an expansion of an ordered group $(M, <, +)$,
then the cut $\langle C, D \rangle$ is said to be valuational 
if there exists $\varepsilon > 0$ such that for all $x\in C$ and $y \in D$, we have $y-x > \varepsilon$.
Otherwise the cut $\langle C, D \rangle$ is said to be non-valuational.
The structure $\cM$ is said to be non-valuational if
all definable cuts in $\cM$ are non-valuational.

We equip $M$ ($\ov{M}$) with the interval topology
(the open intervals form a base),
and each product $M^n$ ($(\ov{M})^n$) with 
the corresponding product topology, respectively.

Recall the notion of definable functions from \cite{W2}.
Let $n$ be a positive integer and $A \subseteq M^n$ definable.
A function $f : A \to \ov{M}$ is said to be definable
if the set 
$\{ \langle x, y \rangle \in M^{n+1} : 
x \in A, y < f(x) \}$ is definable.
A function $f : A \to \ov{M} \cup \{-\infty, +\infty\}$ 
is said to be definable
if $f$ is a definable function from $A$ to $\ov{M}$,
$f(x) = -\infty$ for all $x \in A$, or $f(x) = +\infty$ 
for all $x \in A$.

We recall the notion of strong cells from \cite{W1}.

\begin{definition}
For every $m \in \bN_{+}$,
we define, by induction,  strong cells $C \subseteq M^m$ and their completions $\ov{C} \subseteq \ov{M}^m$.

\begin{enumerate}
\item
Any singleton of $M$ is a $0$-strong cell in $M$ and is equal to its completion.

\item
Any non-empty open convex definable subset of $M$ is a 1-strong cell in $M$.
If $C \subseteq M$ is a 1-strong cell,
then we define its completion by $\ov{C} := \{ x \in \ov{M}  : (\exists a, b \in C)(a < x < b)\}$.

Let $m \in \bN_{+}$, $k \leq m$ and suppose that we have already defined $k$-strong cells in $M^m$ and 
their completions in $\ov{M}^m$.

\item
If $C \subseteq M^m$ is a $k$-strong cell and $f : C \to M$ is a continuous definable function 
which has a unique continuous extension $\ov{f} : \ov{C} \to \ov{M}$,
then the graph of $f$ which is denoted by $\Gamma(f)_C$, is a $k$-strong cell in $M^{m+1}$ and 
its completion in $\ov{M}^{m+1}$ is defined as $\Gamma(\ov{f})_{\ov{C}}$.

\item
If $C \subseteq M^m$ is a $k$-strong cell and $f,g : C \to \ov{M} \cup \{ \pm \infty \}$ is continuous definable functions
such that $f, g$ have continuous extensions $\ov{f}, \ov{g} : \ov{C} \to \ov{M} \cup \{ \pm \infty \}$, where $\ov{f}(\ov{a}) < \ov{g}(\ov{a})$ for any $\ov{a} \in \ov{C}$,
then the set $(f, g)_C := \{ \langle \ov{a}, b \rangle \in C \times M : f(\ov{a}) <b < g(\ov{a}) \}$ is called a $(k+1)$-strong cell in $M^{m+1}$.
The completion of $(f, g)_C$ in $\ov{M}^{m+1}$ is defined as  $(\ov{f}, \ov{g})_{\ov{C}} := \{ \langle \ov{a}, b \rangle \in \ov{C} \times \ov{M} : \ov{f}(\ov{a}) <b < \ov{g}(\ov{a}) \}$.

\item 
Every $k$-strong cell is of the form given in (1) through (4).
We say that $C \subseteq M^m$ is a strong cell in $M^m$ 
if there exists a non-negative integer $k$ such that $C$ is a $k$-strong cell in $M^m$.
\end{enumerate}

\end{definition}

A function $f : C \to \ov{M}$ is said to be a cell-defining function if
$\Gamma(f)_C$ is a strong cell or there is a definable function $g : C \to \ov{M}$
so that $(f, g)_C$ or $(g, h)_C$ is a strong cell.

A strong cell $C \subseteq M^m$, $m\geq 2$, is called a refined strong cell if 
each of the cell-defining functions appearing in its definition 
assumes all its values in one of the following sets: 
$M$, $\ov{M} \setminus M$,  $\{-\infty \}$, or $\{ +\infty\}$.
Refined strong cells in $M$ coincide with strong cells in $M$.
Note that the definitions of strong cells in \cite{W1} and \cite{W2} are not identical.
A strong cell in \cite{W2} is called a refined strong cell in \cite{W1}.
In this paper, we follow the terminology of \cite{W1}.

Let $C \subseteq M^m$ be a strong cell and $f : C \to \ov{M}$ definable.
The function $f$ is said to be strongly continuous 
if it has a unique continuous extension $\ov{f} : \ov{C} \to \ov{M}$.

\begin{definition}
Any finite partition of $M$ into singletons and convex open sets definable in $\cM$ is called a strong cell decomposition of $M$.
A finite partition $\mathcal C$ of $M^{m+1}$ into strong cells is said to be a strong cell decomposition of $M^{m+1}$
if $\pi[\mathcal C] := \{\pi[C] : C \in \mathcal C \}$ is a strong cell decomposition of $M^m$.
Here $\pi : M^{m+1} \to M^m$ denotes the projection dropping the last coordinate.
A strong cell decomposition $\mathcal C$ of $M^m$ partitions $X \subseteq M^m$
if for any $C \in \mathcal C$, we have $C \subseteq X$ or $C \cap X = \emptyset$.
The structure $\cM$ has the strong cell decomposition property
if for any positive integers $m$, $k$ and any definable sets $X_1, \ldots, X_k \subseteq M^m$,
there exists a decomposition of $M^m$ into finitely many strong cells partitioning each of the sets $X_1, \ldots, X_k$.
Similarly, we can define the refined strong cell decomposition property.
\end{definition}

The following result follows from Lemma~2.6, Corollary~2.16, and Corollary~2.17 of \cite{W1}.

\begin{fact} \label{valuational}
Let $\cM = (M, <, +, \ldots)$ be a weakly o-minimal expansion of an ordered group.
Then the structure $\cM$ is non-valuational if and only if it has the refined strong cell decomposition property.
\end{fact}

The following definition appears in \cite{W2}.

\begin{definition}
	Let $\mathcal M = (M, <, \ldots)$ be a weakly o-minimal structure with the refined strong cell decomposition property.
	For any $m \in \N_{+}$ and $i_1, \dots, i_m \in \{0,1\}$, we define basic $\langle i_1, \dots, i_m \rangle$-cells in $\overline{M}^m$ inductively.
	
	If $1 \leq j_1 < \dots < j_k \leq m$, let
	\[
	\rho^m_{j_1, \dots, j_k} : \overline{M}^m \to \overline{M}^k
	\]
	denote the projection onto the coordinates $j_1, \dots, j_k$.
	When $m$ is clear from the context, we simply write $\rho_{j_1, \dots, j_k}$.
	
	\begin{enumerate}
		\item
		A singleton of $\overline{M}^m$ is called a basic $\langle 0, \dots, 0 \rangle$-cell in $\overline{M}^m$.
		
		\item
		If $C \subseteq M$ is a strong $\langle 1 \rangle$-cell, then $\overline{C}$ is called a basic $\langle 1 \rangle$-cell.
		Note that
		\[
		\rho_1(\overline{C}) \cap M = \overline{C} \cap M = C
		\]
		is an open strong cell in $M$.
		
		\item
		If $C = \{ \overline{a} \} \subseteq \overline{M}^m$ and $I$ is a basic $\langle 1 \rangle$-cell in $\overline{M}$,
		then $C \times I$ is called a basic $\langle 0, \dots, 0, 1 \rangle$-cell in $\overline{M}^{m+1}$.
		Note that
		\[
		\rho_{m+1}(C \times I) \cap M = I \cap M
		\]
		is an open strong cell in $M$.
		
		\item[]
		Assume that $i_1, \dots, i_m \in \{0,1\}$ with $i_1 + \dots + i_m > 0$, and suppose that basic
		$\langle i_1, \dots, i_m \rangle$-cells in $\overline{M}^m$ have already been defined.
		Let
		\[
		\{ j_1, \dots, j_k \} = \{ j \in \{1, \dots, m\} : i_j = 1 \},
		\qquad j_1 < \dots < j_k,
		\]
		and suppose that for every basic $\langle i_1, \dots, i_m \rangle$-cell
		$C \subseteq \overline{M}^m$, the set
		\[
		\rho_{j_1, \dots, j_k}(C) \cap M^k
		\]
		is an open strong cell in $M^k$.
		
		\item
		Let $C \subseteq \overline{M}^m$ be a basic $\langle i_1, \dots, i_m \rangle$-cell, and set
		\[
		D := \rho_{j_1, \dots, j_k}(C) \cap M^k,
		\]
		which is an open strong cell in $M^k$.
		If $f$ is a strongly continuous definable function from $D$ to $M$
		or from $D$ to $\overline{M} \setminus M$,
		then
		\[
		\Gamma\bigl(\overline{f} \circ \rho_{j_1, \dots, j_k}\!\restriction_C \bigr)
		\]
		is called a basic $\langle i_1, \dots, i_m, 0 \rangle$-cell in $\overline{M}^{m+1}$.
		Note that
		\[
		\rho_{j_1, \dots, j_k}\!\left(
		\Gamma\bigl(\overline{f} \circ \rho_{j_1, \dots, j_k}\!\restriction_C \bigr)
		\right) \cap M^k = D
		\]
		is an open strong cell in $M^k$.
		
		\item
		Let $C \subseteq \overline{M}^m$ be a basic $\langle i_1, \dots, i_m \rangle$-cell, and set
		\[
		D := \rho_{j_1, \dots, j_k}(C) \cap M^k,
		\]
		an open strong cell in $M^k$.
		If $f, g : D \to \overline{M} \cup \{ \pm \infty \}$ are strongly continuous definable functions such that
		all values of $f$ and $g$ lie in one of the sets
		\[
		\{ -\infty \}, \quad M, \quad \overline{M} \setminus M, \quad \{ +\infty \},
		\]
		and
		\[
		\overline{f}(\overline{x}) < \overline{g}(\overline{x})
		\quad \text{for all } \overline{x} \in \overline{D},
		\]
		then the set
		\[
		(\overline{f} \circ \rho_{j_1, \dots, j_k}, \overline{g} \circ \rho_{j_1, \dots, j_k})_C
		:= \left\{
		\langle \overline{a}, b \rangle \in C \times \overline{M} :
		(\overline{f} \circ \rho_{j_1, \dots, j_k})(\overline{a}) < b <
		(\overline{g} \circ \rho_{j_1, \dots, j_k})(\overline{a})
		\right\}
		\]
		is called a basic $\langle i_1, \dots, i_m, 1 \rangle$-cell in $\overline{M}^{m+1}$.
		Note that
		\[
		\rho_{j_1, \dots, j_k}
		\left[
		(\overline{f} \circ \rho_{j_1, \dots, j_k},
		\overline{g} \circ \rho_{j_1, \dots, j_k})_C
		\right]
		\cap M^{k+1}
		= (f, g)_D
		\]
		is an open strong cell in $M^{k+1}$.
	\end{enumerate}
	
	In the standard way, we define a basic cell decomposition of $\overline{M}^m$ as a finite partition of $\overline{M}^m$ into basic cells.
\end{definition}

For any $n \in \bN_{+}$ and any refined strong cell $C \subseteq M^n$,
we denote by $R_C$ an $n$-ary relational symbol.
We interpret $R_C$ in $\ov{M}^n$ as $\ov{C}$, the completion of $C$.
According to Section 2 of \cite{W2}, 
the structure $\ov{\mathcal M} := (\ov{M}, <, (R_C: C \text{ is a refined strong cell}))$ is o-minimal 
and is called the canonical o-minimal extension of $\cM$.
If $X \subseteq \ov{M}^m$ is  a set definable in $\ov{\cM}$, then $X \cap M^m$ is definable in $\cM$.
If additionally $Y \subseteq M^m$ is definable in $\cM$, then $X \cap Y$ is definable in $\cM$.

The following result follows from Proposition~2.3 of \cite{W2}.

\begin{fact} \label{basic}
	Assume that $\mathcal M = (M, <, \dots)$ is a weakly o-minimal structure with the strong cell decomposition property and $\ov{\mathcal M} = (\ov{M}, <, \dots)$ is its canonical o-minimal extension.
	Let $m, k \in \N_{+}$.
	If $X_1, \dots, X_k \subseteq \ov{M}^m$ are sets definable in $\ov{\M}$,
	then there exists a basic cell decomposition of $\ov{M}^m$ into finitely many basic cells partitioning each of the sets $X_1, \ldots, X_k$.
\end{fact}


\begin{proposition} \label{open}
Suppose that $\mathcal M = (M, <, \ldots)$ is a weakly o-minimal structure with the refined strong cell decomposition property.
Let $n \in \bN_{+}$.
Suppose that  $X$ is a definable open subset of $M^n$.
Then there exists an $\overline{\M}$-definable open set $Y \subseteq \overline{M}^n$ such that
$X = Y \cap M^n$.
\end{proposition}

\begin{proof}
	Take a refined strong cell decomposition of $X$, so that
	\[
	X = C_1 \cup \dots \cup C_k.
	\]
	By Fact~\ref{basic},
	there exists a basic cell decomposition $\mathcal{C}$ of $\overline{M}^n$ partitioning each of the sets $\ov{C_1},\dots, \ov{C_k}$.
	Define
	\[
	\hat{X}
	= \bigcup \left\{ D\in \mathcal{C} : \exists i\ (D \subseteq \overline{C_i}) \right\}
	\;\cup\;
	\bigcup \left\{ D\in \mathcal{C} : \bigl(\forall i\ (D \cap \overline{C_i} = \emptyset)\bigr) \wedge (D \cap M^n = \emptyset) \right\}.
	\]
	Set $Y = \Int_{\overline{M}}(\hat{X})$.
	Since $\mathcal{C}$ is finite, $\hat{X}$ is definable in $\overline{\M}$; hence $Y$ is definable  in $\overline{\M}$ as well.
	We prove that $X = Y \cap M^n$.
	
	By the definition of $\hat{X}$, the inclusion $X \supseteq Y \cap M^n$ is clear.
	
	It remains to show $X \subseteq Y \cap M^n$.
	Assume for a contradiction that there exists $a=(a_1,\dots,a_n)\in X$ such that
	$a \notin \Int_{\overline{M}}(\hat{X})$.
	Since $X$ is open in $M^n$, there exists an open box $U_0 \subseteq X$ in $M^n$ containing $a$.
	Since $a \notin \operatorname{int}_{\overline{M}}(\hat{X})$, there exists a point $b_0 \in \overline{U_0}$ such that $b_0 \notin \hat{X}$.
	Let $U_1$ be an open box in $M^n$ containing $a$ such that
	$\cl_M(U_1) \subseteq U_0$ and $b_0 \notin \overline{U_1}$.
	Then there exists a point $b_1 \in \overline{U_1}$ such that $b_1 \notin \hat{X}$.
	Iterating this argument, there are infinitely many points
	$b_i \in \overline{U_0}$ ($i\in\omega$) with $b_i \notin \hat{X}$.
	Let $B = \{ b_i : i \in \omega \}$.
	Since $\mathcal{C}$ is finite, there exists $C \in \mathcal{C}$ such that infinitely many of the $b_i$ lie in $C$.
	Replacing $B$ by this infinite subset if necessary, we may assume $B \subseteq C$.
	Moreover, since $a \notin \Int_{\overline{M}}(\hat{X})$, for every open box $U \subseteq X$ in $M^n$ containing $a$ there exists $b \in \overline{U}$ with $b \notin \hat{X}$;
	hence we may assume $a \in \cl_{\overline{M}}(C)$.
	Also, $C$ is infinite.
	
	If $C \subseteq \overline{C_i}$ for some $i$, then $B \subseteq \overline{C_i} \subseteq \hat{X}$, a contradiction.
	Thus for every $i$ we have $C \cap \overline{C_i} = \emptyset$, and hence $C \cap X = \emptyset$.
	
	First, consider the case where $C$ is a basic $\langle 0,\dots,0,1 \rangle$-cell.
	Then there exists an open interval $I \subseteq \overline{M}$ such that
	\[
	C = (a_1,\dots,a_{n-1}) \times I.
	\]
	Since $a \in X \subseteq M^n$ and $a \in \cl_{\overline{M}}(C)$, every open box $U$ in $M^n$ containing $a$ satisfies
	$C \cap U \ne \emptyset$, equivalently $(C \cap M^n) \cap U \ne \emptyset$.
	Because $C \cap X = \emptyset$, we have $(C \cap M^n) \cap X = \emptyset$.
	Hence $a \notin \Int_M(X)$, contradicting that $X$ is open in $M^n$.
	
	Next, consider the case where $C$ is a basic $\langle i_1,\dots,i_{n-1},0 \rangle$-cell and $i_1+\dots+i_{n-1} > 0$.
	That is, there exists a basic $\langle i_1,\dots,i_{n-1} \rangle$-cell $\hat{D} \subseteq \overline{M}^{n-1}$, and if we set
	\[
	\{j_1,\ldots,j_k\} = \{\, j \in \{1,\dots,n-1\} : i_j = 1 \,\},
	\qquad j_1 < \dots < j_k,
	\]
	then
	\[
	D = \rho_{j_1,\ldots,j_k}(\hat{D}) \cap M^k
	\]
	is an open strong cell in $M^k$.
	Moreover, there exists a strongly continuous definable function $f$ from $D$ to $M$ or from $D$ to $\overline{M}\setminus M$ such that
	\[
	C = \Gamma\bigl(\overline{f} \circ \rho_{j_1,\ldots,j_k}(\hat{D})\bigr).
	\]
	
	Assume that $f:D\to M$.
	Recall that $C \cap \overline{C_i} = \emptyset$ for every $1 \leq i \leq k$.
	If $C \cap M^n = \emptyset$, $C \subseteq \hat{X}$ by the definition of $\hat{X}$.
	Hence $B \subseteq C \subseteq \hat{X}$, a contradiction.
	
	Assume that $C \cap M^n \ne \emptyset$.
	In the construction of a function-type cell in the basic cells, all cell-defining functions take values in $M$.
	We claim that $a \in \cl_{\overline{M}}(C \cap M^n)$.
	Let $U \subseteq M^n$ be an arbitrary open box containing $a$.
	Since $a \in \cl_{\overline{M}}(C)$, there exists a point $x=(x_1,\dots,x_n) \in \overline{U} \cap C$.
	Set $\tilde{C} = \rho_{j_1,\ldots,j_k}(C)$ and $V = \rho_{j_1,\ldots,j_k}(U)$.
	Then $\tilde{C}$ is open in $\overline{M}^k$, and $V$ is an open box in $M^k$.
	Let $\tilde{x} = (x_{j_1},\dots,x_{j_k})$.
	Then $\tilde{x} \in \overline{V} \cap \tilde{C}$.
	Since $\tilde{C}$ is open in $\overline{M}^k$, there exists an open box $W$ in $M^k$ such that
	$\tilde{x} \in \overline{W} \subseteq \overline{V} \cap \tilde{C}$.
	Choose $\tilde{y} \in W$.
	Since $f:D\to M$, there exists $y \in (C \cap M^n) \cap U$ such that $\rho_{j_1,\ldots,j_k}(y) = \tilde{y}$.
	Therefore $(C \cap M^n) \cap U \ne \emptyset$, and hence $a \in \cl_{\overline{M}}(C \cap M^n)$.
	However, since $C \cap X = \emptyset$, we have $(C \cap M^n) \cap X = \emptyset$, and hence $a \notin \Int_M(X)$.
	This contradicts the assumption that $X$ is open in $M^n$.
	
	Assume instead that $f:D\to \overline{M}\setminus M$.
	Then $C \cap M^n = \emptyset$, and similarly $C \subseteq \hat{X}$.
	Hence $B \subseteq C \subseteq \hat{X}$, a contradiction.
	
	Therefore the function-type case leads to a contradiction.
	
	Finally, consider the case where $C$ is a basic $\langle i_1,\dots,i_{n-1},1 \rangle$-cell and $i_1+\dots+i_{n-1} > 0$.
	As before, there exist a basic $\langle i_1,\dots,i_{n-1} \rangle$-cell $\hat{D} \subseteq \overline{M}^{n-1}$ and an open strong cell $D \subseteq M^k$.
	Moreover, there exist strongly continuous functions $f,g : D \to \overline{M} \cup \{\pm\infty\}$ such that all values of $f$ and $g$ lie in exactly one of
	$\{-\infty\},\ M,\ \overline{M}\setminus M,\ \{\infty\}$, and
	$\overline{f}(x) < \overline{g}(x)$ for any $x \in \overline{D}$, and
	\[
	C = (\overline{f} \circ \rho_{j_1,\ldots,j_k},\ \overline{g} \circ \rho_{j_1,\ldots,j_k})_{\hat{D}}.
	\]
		Recall that $C \cap \overline{C_i} = \emptyset$ for every $1 \leq i \leq k$.
	If $C \cap M^n = \emptyset$, $C \subseteq \hat{X}$ by the definition of $\hat{X}$.
	Hence $B \subseteq C \subseteq \hat{X}$, a contradiction.

	Assume $C \cap M^n \ne \emptyset$.
	As in the previous case, we have $a \in \cl_{\overline{M}}(C \cap M^n)$.
	Hence every open box $U$ in $M^n$ containing $a$ meets $C \cap M^n$.
	Because $C \cap X = \emptyset$, we have $(C \cap M^n) \cap X = \emptyset$, so $a \notin \Int_M(X)$, contradicting openness of $X$ in $M^n$.
	This completes the proof.
\end{proof}



We recall the notion of open cell property from \cite{A}.
An o-minimal structure $\mathcal M = (M, <, \dots)$ is said to have the open cell property
if every non-empty definable open subset of $M^n$ is a union of finitely many open cells.
A weakly o-minimal structure $\mathcal M = (M, <, \dots)$ is said to have the open cell property
if every non-empty definable open subset of $M^n$ is a union of finitely many open strong cells.

The following example shows that the open cell property fails for refined strong cells.

\begin{example}
Let $R_{\mathrm{alg}}$ denote the ordered field of real algebraic numbers 
and let the unary predicate symbol $P$ be interpreted by the convex set $S = (-\pi, \pi) \cap R_{\mathrm{alg}}$.
Then, by Proposition 2.1 of \cite{MMS},
the structure $\mathfrak R = (R_{\mathrm{alg}}, <, +, \cdot, P)$ is non-valuational weakly o-minimal.
So the structure $\mathfrak R$ has the refined strong cell decomposition property.
Consider the functions $f, g : I=(-1, 1) \to \ov{R}_{\mathrm{alg}}$ defined by $f(x) = -4$ and $g(x) = \pi x$.
The set $(f, g)_I$ is an open strong cell but not a refined strong cell
because $g(0) = 0 \in R_{\mathrm{alg}}$ and $g(x) \notin R_{\mathrm{alg}}$
for $x \neq 0$.
Moreover, $(f, g)_I$ cannot be expressed as a finite union of open refined strong cells.
\end{example}


\begin{fact}[{\cite[Theorem~1.3]{Wi}}] \label{Wi}
Suppose that  $\cM = (M, <, +, \cdot, \ldots)$ is an o-minimal expansion of a real closed field.
Let $U$ be a definable bounded open subset of $M^n$.
Then, there exists a finite collection of open  cells in $M^n$ whose union is $U$.
\end{fact}

\begin{theorem}
Suppose that  $\cM = (M, <, +, \cdot, \ldots)$ is a non-valuational weakly o-minimal expansion of a real closed field.
Let $U$ be a definable bounded open subset of $M^n$.
Then, there exists a finite collection of open strong cells in $M^n$ whose union is $U$.
\end{theorem}

\begin{proof}
By Fact~\ref{valuational}, the structure $\mathcal M$ has the refined strong cell decomposition property.
Since $U$ is a definable bounded open subset of $M^n$, by Proposition~\ref{open},
there exists an $\overline{\M}$-definable bounded open set $V \subseteq \overline{M}^n$ such that
$U = V \cap M^n$.
By Fact~\ref{Wi}, there exist open cells $D_1, \dots, D_l$ such that
$V = D_1 \cup \dots \cup D_l$.
Because each $D_i \cap M^n$ is an open strong cell and ${U} = (D_1 \cap M^n) \cup \dots \cup (D_l \cap M^n)$,
this completes the proof.
\end{proof}

\section{Weakly o-minimal expansions of ordered groups}

Let $\mathcal M = (M, <, +, \cdots)$ be a non-valuational weakly o-minimal expansion of an ordered group.
If $\M$ is o-minimal, we say that $\mathcal M$ is semi-bounded if
there exists no definable bijection between a bounded interval and an unbounded interval \cite{E}.

\begin{theorem}[\cite{EEP}]
Let $\mathcal M = (M, <, +, \cdots)$ be a semi-bounded o-minimal expansion of an ordered group.
Then every non-empty definable open set is a finite union of open cells.
\end{theorem}

\begin{theorem}
Consider a non-valuational weakly o-minimal expansion $\mathcal{M}=(M, <, +, \dots)$ of an ordered group.
Let $U$ be a definable open subset of $M^n$ and assume that $\overline{\M}$ is semi-bounded.
Then there exist finitely many open strong cells whose union is $U$.
\end{theorem}

\begin{proof}
By Proposition~\ref{open}, there exists an $\overline{\M}$-definable open subset $V$ of $\overline{M}^n$ such that 
$U=V \cap M^n$.
By Theorem 3.1, there exist open cells $D_1, \dots, D_l$ such that
$V= D_1 \cup \dots \cup D_l$.
Since each $D_i \cap M^n$ is an open strong cell and ${U} = (D_1 \cap M^n) \cup \dots \cup (D_l \cap M^n)$,
this completes the proof.
\end{proof}

We obtained an improvement of Theorem~3.2
inspired by a private communication with Fujita \cite{F}.

\begin{theorem}
Let $\mathcal{M}=(M, <, +, \dots)$ be a non-valuational weakly o-minimal expansion of an ordered group.
Then the following three conditions are equivalent.

(1)   
For any bounded definable set $I \subseteq M$ and for every $\mathcal{M}$-definable function $f:I \to \overline{M}$,
$f(I)$ is bounded.

(2) The canonical o-minimal extension $\overline{\mathcal{M}}$ of $\mathcal{M}$ is semi-bounded.

(3) Any $\mathcal{M}$-definable open set is a finite union of open strong cells.
\end{theorem}

\begin{proof}
$(1) \Rightarrow (2)$
Assume that (2) does not hold.
Then we have a bounded open interval $J_1$, an unbounded interval $J_2$ and 
an $\overline{\M}$-definable bijection $g:J_1 \to J_2$.
Since $\overline{\M}$ is o-minimal and by the monotonicity theorem,
we may assume $g$ is continuous and strictly monotone.
Set $I=J_1 \cap M$, $f=g|I$.
Then $f$ is $\mathcal M$-definable.
 Since $M$ is dense in $\ov{M}$,
 $f(I)$ is unbounded.
This contradicts (1).

$(2) \Rightarrow (1)$
We show the contradiction.
Suppose that $I$ is a bounded subset of $M$ and $f:I \to \overline{M}$ is a function definable in $\mathcal{M}$.
Assume that $f(I)$ is unbounded.
By Proposition~3.1 of \cite{W2}, 
we can decompose $I$ into $I=X \cup J_1 \cup \dots \cup J_k$ such that 
$X$ is finite, each $J_i$ is open convex set and $f|J_i$ is monotone and strongly continuous.
We have the fact that some $f(J_i)$ is unbounded.
Replacing $f$ by $f|J_i$, 
we can assume that $f$ is  monotone and strongly continuous.
Then the unique extension $\overline{f}:\overline{I} \to \overline{M}$ is monotone and continuous.
This is an $\overline{\M}$-definable bijection between a bounded open interval $\overline{I}$ and an unbounded open interval $\overline{f}(\overline{I})$.
This is a contradiction.

$(2) \Rightarrow (3)$
%
By Theorem 3.2, this implication follows.
%
%
%

$(3) \Rightarrow (2)$
Assume that (2) does not hold.
Then we have an $\overline{\M}$-definable bijection $f:J_1 \to J_2$ between a bounded open interval $J_1$ and an unbounded interval $J_2$.
Since $\overline{\M}$ is o-minimal and by the monotonicity theorem,
we may assume that $f$ is strictly monotone and continuous.
We may further assume that $J_2$ is bounded below and $f$ is strictly decreasing.
By shifting, we can set $J_1=(0, b)$.
Let $Y$ be the $\overline{\M}$-definable open set

$Y=\{(x, y) \in \overline{M}^2 : (-b<x<b) \land ((x<0) \to (y<f(-x))) \land ((x>0) \to (y<f(x)))\}$.

Set $X=Y \cap M^n$.
Then $X$ is $\mathcal{M}$-definable.
We prove that $X$ is not a finite union of open strong cells.

Assume $X$ is a finite union of open strong cells.
Note that $\{ 0 \} \times M$ is contained in $X$.
Thus there exists an open strong cell $C$ and $r \in M$ such that 
$\{0\} \times (r, \infty) \subseteq C \subseteq X$.
Since $\{0\} \times (r,  \infty) \subseteq C$,
we can find a definable open convex set $I$ with $0 \in I$ and a strongly continuous definable function $g:I \to \overline{M} \cup \{-\infty\}$
such that 
$C=\{ (x, y) \in M^2 : x \in I, y>g(x)\}$.
Hence $C$ contains a point $(x_0, y_0)$ with $0<x_0<b$ and $y_0>f(x_0)$.
It contradicts $C \subseteq X$.
\end{proof}


\begin{thebibliography}{99}
%



\bibitem{A} 
S.~Andrews, Definable open sets as finite unions of definable open cells,
Notre Dame J. Form. Log. {\bf 51} (2010), no. 2, 247-251.

\bibitem{C}
M.~Coste, 
An introduction to o-minimal geometry,
Dottorato di Ricerca in Matematica, Dip.\ Mat.\ Univ.\ Pisa, 
Istituti Editoriali e Poligrafici Internazionali (2000).


\bibitem{Di} 
M.~A.~Dickmann, Elimination of quantifiers for ordered valuation rings, 
J. Symb. Log. {\bf 52} (1987), 116-128.

\bibitem{D} 
L.~van den Dries, Tame topology and o-minimal structures, London Mathematical Society Lecture
Notes Series, 248, Cambridge University Press, Cambridge, (1998).

\bibitem{E} 
M.~Edmundo,
Structure theorems for o-minimal expansions of groups,
Ann. Pure Appl. Logic {\bf 102} (2000), no. 1-2, 159-181.


\bibitem{EEP} 
M.~Edmundo, P.~Eleftheriou and L.~Prelli,
Coverings by open cells,
Arch. Math. Logic {\bf 53} (2014), 307-325.


\bibitem{ET} 
J. S.~Eivazloo and S.~Tari,
SCE-cell decomposition and OCP in weakly o-minimal structures,
Notre Dame J. Form. Log. {\bf 57} (2016), 399-410.



\bibitem{F}
M.~Fujita, Private communication, (2025).



\bibitem{MMS} 
D.~Macpherson, D.~Marker and C.~Steinhorn, Weakly o-minimal structures and real closed fields, Trans.
Amer. Math. Soc. \textbf{352} (2000), 5435-5483.






\bibitem{W1} 
R.~Wencel,  Weakly o-minimal non-valuational structures,
Ann. Pure Appl. Logic \textbf{154} (2008), 139-162.


\bibitem{W2} 
R.~Wencel, On the strong cell decomposition property for weakly o-minimal structures, Math. Log.
Quart. \textbf{59} (2013), 452-470.

\bibitem{Wi} 
A.~Wilkie, Covering open definable sets by open cells,
In: M.~Edmundo, D.~Richardson, A.~J.~Wilkie
(eds.) O-minimal Structures, Proceedings of the RAAG Summer School Lisbon 2003, Lecture Notes
in Real Algebraic and Analytic Geometry, Cuvillier Verlag, (2005).


\end{thebibliography}
\end{document}